\documentclass[a4paper]{article}
\AtBeginDocument{%
  \providecommand\BibTeX{{%
    \normalfont B\kern-0.5em{\scshape i\kern-0.25em b}\kern-0.8em\TeX}}}

\usepackage{enumitem}
\usepackage[utf8]{inputenc}
\usepackage{mathtools}
\usepackage{amsfonts}
\usepackage{amsthm}
\usepackage{amsmath}
\usepackage{amssymb}
\usepackage{tikz}
\usetikzlibrary{trees}
\usepackage{hyperref}
\usepackage{enumitem}

\newtheorem{Problem}{Problem}

\newtheorem{Remark}{Remark}

\newtheorem{Proposition}{Proposition}
\newtheorem{Lemma}{Lemma}
\newtheorem{Conjecture}{Conjecture}
\newtheorem{Corollary}{Corollary}
\newtheorem{Algorithm}{Algorithm}
\newtheorem{Definition}{Definition}
\newtheorem{Example}{Example}

\newcommand{\uproman}[1]{\uppercase\expandafter{\romannumeral#1}}

\usepackage{pgfplots}\pgfplotsset{compat=1.13}
\usetikzlibrary{arrows,positioning,calc,patterns}

\def\eatspace#1{#1}
\def\step#1#2{\par\kern1pt\dimen44=#2em\advance\dimen44 1.67em\hangindent\dimen44\hangafter=1\noindent\rlap{\small#1}\kern\dimen44\relax\eatspace}

\def\<#1>{\langle#1\rangle}

\def\val{\operatorname{val}}

\AtBeginDocument{%
  \providecommand\BibTeX{{%
    Bib\TeX}}}

\begin{document}

\title{Separated Variables on Plane Algebraic Curves}

\author{Manfred Buchacher}

\maketitle

\section{Abstract}
We investigate the problem of deciding whether the restriction of a rational function $r\in\mathbb{K}(x,y)$ to the curve associated with an irreducible polynomial $p\in\mathbb{K}[x,y]$ is the restriction of an element  of $\mathbb{K}(x)+\mathbb{K}(y)$. We present an algorithm and a conjectural semi-algorithm for finding such elements depending on whether $p$ has a non-trivial rational multiple in $\mathbb{K}(x) + \mathbb{K}(y)$ or not.

%
\section{Introduction}

The following is the continuation of the work on an elimination problem that was started in~\cite{buchacher2020separating} and extended in~\cite{buchacher2024separating}. The latter discussed a semi-algorithm that takes an irreducible polynomial of $\mathbb{K}[x,y]$ as input and outputs a description of all its non-trivial rational multiples in $\mathbb{K}(x) + \mathbb{K}(y)$. Building upon this work, we present a semi-algorithm and a conjectural semi-algorithm which for any $r\in\mathbb{K}(x,y)$ and any irreducible polynomial $p\in\mathbb{K}[x,y]$ find an element of $\mathbb{K}(x) + \mathbb{K}(y)$ that is of the form $r + qp$ for some $q\in\mathbb{K}(x,y)$ whose denominator is not divisible by $p$. As in~\cite{buchacher2024separating}, the semi-algorithms may not terminate. Their termination depends on a dynamical system on the curve associated with $p$ and the location of the poles of $r$ thereon.

This work originates from the study of linear discrete differential equations, a class of functional equations that arise in the enumeration of restricted lattice walks. Their systematic study was initiated in~\cite{bousquet2010walks, mishna2009classifying}, and they have attracted considerable attention since then. Contributions come from diverse fields such as (computer) algebra, complex analysis, the theory of boundary value problems, the Galois theory of linear difference equations and probability theory. We refer to~\cite{dreyfus2018nature, dreyfus2020walks} and the references therein for an overview of the relevant literature. Of particular relevance for this work is~\cite{bernardi2020counting}. It explains how partial discrete differential equations can be reduced to ordinary discrete differential equations. The reduction relies on the existence of certain rational functions, so-called invariants and decoupling functions. Whether they exist, and in case they do, how to construct them, are therefore important questions. Though partially answered in~\cite{bernardi2020counting}, it remained an open problem how they can be constructed in general. 
An extension of the ideas of~\cite{bernardi2020counting} are found in~\cite{hardouin}. It proposes a semi-algorithm for computing invariants, and an algorithm for computing decoupling functions when there are (non-trivial) invariants. Another approach was followed in~\cite{buchacher2024separating}. It develops a semi-algorithm for constructing invariants based on the ideas developed in~\cite{buchacher2020separating}. Yet, the problem of how to determine decoupling functions was not addressed. This is the purpose of this paper.

The paper is organized as follows. In Section~\ref{sec:pre} we recall some definitions and facts that will be used in the rest of the paper. We also make precise what this article is about and state the problem we will solve. The general strategy it is addressed with is presented in Section~\ref{sec:strat}. In Section~\ref{sec:alg} we formulate the algorithm when there are invariants and prove its correctness and completeness. We also discuss a conjectural semi-algorithm that addresses the other case. The problems in algebraic dynamics they give rise to are presented in Section~\ref{sec:problems}. 

\section{Preliminaries}\label{sec:pre}

We assume throughout that $\mathbb{K}$ is an algebraically closed field of characteristic $0$. We denote by $\mathbb{K}[x,y]$ the ring of polynomials in $x$ and $y$ over $\mathbb{K}$, and we write $\mathbb{K}(x,y)$ for its quotient field. Given a rational function $r\in\mathbb{K}(x,y)$ in reduced form, we write $r_n$ and $r_d$ for its numerator and denominator, respectively. Conversely, given two coprime polynomials $r_n, r_d\in\mathbb{K}[x,y]$, we denote their quotient $r_n/r_d$ by $r$. We denote the total degree of a polynomial $p\in\mathbb{K}[x,y]$ by $\deg p$, and we write $\val p$ for its valuation, that is, the degree of its terms of lowest degree. Given $f\in\mathbb{K}(x)$, and $s\in\mathbb{K}\cup \{\infty\}$, we let $\mathrm{m}(s,f)$ be the negative of the valuation of $f$ when considered as a series in $x-s$ or $x^{-1}$ depending on whether $s$ is finite or not. If $s = \infty$, then $\mathrm{m}(s,f)$ is simply $\deg f_n - \deg f_d$. If $s\in\mathbb{K}$ is a pole of $f$, then $\mathrm{m}(s,f)$ is its multiplicity. If $s\in\mathbb{K}$ is a root, then $\mathrm{m}(s,f)$ is the negative of its multiplicity. 

\begin{Definition}
Let $p\in\mathbb{K}[x,y]\setminus\left( \mathbb{K}[x] \cup \mathbb{K}[y]\right)$ be irreducible. We write $\mathbb{K}[x,y]_p$ for the set of elements of $\mathbb{K}(x,y)$ whose denominator is not divisible by $p$. It is closed under addition and multiplication, and hence forms a ring. It is the \textbf{local ring} of $\mathbb{K}[x,y]$ at $p$. The polynomial $p$ is an element of $\mathbb{K}[x,y]_p$. We denote the ideal it generates therein by $\langle p\rangle$. It consists of all elements of $\mathbb{K}(x,y)$ whose numerator is a multiple of $p$. 
\end{Definition}


The problem this article is about can now be stated as follows.
\begin{Problem}\label{prob:sep}
Given a rational function $r\in\mathbb{K}(x,y)$ and an irreducible polynomial $p\in\mathbb{K}[x,y]$ that is not univariate, compute
\begin{equation*}
\left( r + \langle p \rangle \right)  \cap \left( \mathbb{K}(x) + \mathbb{K}(y) \right).
\end{equation*}
\end{Problem}
In order to solve Problem~\ref{prob:sep}, it is convenient to introduce the set 
\begin{equation*}
\mathrm{F}(r,p) := \left\{ (f,g)\in\mathbb{K}(x)\times\mathbb{K}(y) : f - g \in r + \langle p \rangle \right\}
\end{equation*}
and clarify its algebraic structure. The following proposition is immediate.
\begin{Proposition}\label{prop:structure}
Let $(f,g)$ be any element of $\mathrm{F}(r,p)$. Then
\begin{equation*}
\mathrm{F}(r,p) = (f,g) + \mathrm{F}(0,p).
\end{equation*}
\end{Proposition}
For notational convenience, and in accordance with~\cite{buchacher2024separating}, we write $\mathrm{F}(p)$ for $\mathrm{F}(0,p)$. It was shown in~\cite[Proposition~1, 2]{buchacher2024separating} that $\mathrm{F}(p)$ is a simple field. 

\begin{Proposition}
Let $p\in\mathbb{K}[x,y]\setminus\left( \mathbb{K}[x] \cup \mathbb{K}[y]\right)$ be irreducible.
Then $\mathrm{F}(p)$ is a field with respect to component-wise addition and multiplication. It is referred to as the \textbf{field of separated multiples} of $p$. Furthermore, 
\begin{equation*}
\mathrm{F}(p) = \mathbb{K}((f,g))
\end{equation*} 
for some $(f,g)\in\mathbb{K}(x)\times \mathbb{K}(y)$. That is, any element of $\mathrm{F}(p)$ is a rational function in $(f,g)$ with respect to component-wise addition and multiplication.
\end{Proposition}



We will determine an element of $\mathrm{F}(r,p)$ by performing a local and global study of $r$ and $p$. The local study relies on the use of Puiseux series and suitably chosen gradings of $\mathbb{K}[x,y]$. We collect some facts in the following paragraphs. For details, in particular for proofs, we refer to~\cite{walker1950algebraic, kirwan1992complex, fischer2013ebene}. 

\begin{Definition}
A Puiseux series over $\mathbb{K}$ in (positive powers of) $x$ is a (possibly infinite) formal sum of the form 
\begin{equation*}
a_1 x^{\alpha_1} + a_2 x^{\alpha_2} + a_3 x^{\alpha_3} + \dots
\end{equation*}
where $a_1$, $a_2$, $a_3,\dots$ are elements of $\mathbb{K}$ and $\alpha_1 < \alpha_2 < \alpha_3 < \dots$ are rational numbers that have a common denominator. The set of Puiseux series in $x$ is denoted by $\mathbb{K}\{\{x\}\}$. It forms a field with respect to addition and multiplication that is algebraically closed. 
\end{Definition}

An equation $p(x,y) = 0$ can be solved for~$y$ in terms of series in $x$ using the Newton-Puiseux algorithm. Its solutions give rise to local parametrizations of the associated curve and allow one to study the local behavior of elements of its function field. 
\begin{Definition}
Let $(\infty,\infty)$ be a point on the curve in $\mathbb{P}^1(\mathbb{K})\times \mathbb{P}^1(\mathbb{K})$ associated with the bi-homogenization
\begin{equation*}
x_0^{\deg_x p}y_0^{\deg_y p}p(x_1/x_0,y_1/y_0)
\end{equation*}
 of $p$. It is said to be a \textbf{pole} of $r$, if there is a $\varphi\in\mathbb{K}\{\{x^{-1}\}\}$ such that 
\begin{equation*}
p(x,\varphi) = 0, \quad \deg \varphi > 0 \quad \text{and} \quad \deg r(x,\varphi) > 0.
\end{equation*} 
More generally,  a point $(s_1,s_2)$ is said to be a pole, if there are $a_i$, $b_i$, $c_i$, $d_i\in\mathbb{K}$ such that 
\begin{equation*}
T: (x,y) \mapsto \left( \frac{a_1x + b_1}{c_1x+d_1}, \frac{a_2 x + b_2}{c_2 x +d_2} \right)
\end{equation*}
is a transformation of $\mathbb{P}^1(\mathbb{K})\times \mathbb{P}^1(\mathbb{K})$ that maps $(s_1,s_2)$ to $(\infty,\infty)$ and the latter is a pole of $r\circ T^{-1}$ on the curve associated with (the numerator of) $p\circ T^{-1}$ in the sense above.
\end{Definition} 

When studying a polynomial on a curve, one can observe that certain terms contribute to its local behavior more than others. Discarding the latter, the polynomial which remains is homogeneous. 

\begin{Definition}
Let $\omega = (\omega_x, \omega_y)$ be an element of $\mathbb{Z}^2$, and let $ax^iy^j$ be a non-zero term in $\mathbb{K}[x,y]$. We define
\begin{equation*}
\omega(ax^i y^j) := \omega_x i + \omega_y j
\end{equation*}
and call it the weight of $ax^iy^j$ with respect to $\omega$. The weight of a polynomial $p$ is the maximum of the weights of its terms and denoted by $\omega(p)$. Its \textbf{leading part} is the sum of terms of maximal weight. We write $\mathrm{lp}_\omega(p)$ for it. We will later extend these notions to (some) rational functions $r$ by defining $\omega(r) = \omega(r_n) - \omega(r_d)$ and $\mathrm{lp}_\omega(r) = \mathrm{lp}_\omega(r_n) / \mathrm{lp}_\omega(r_d)$.
\end{Definition}

When relating local and global properties of $r$ and $p$ the notion of orbit of a point on a curve will be important. 

\begin{Definition}
Let $p\in\mathbb{K}[x,y]$, and let $C$ be the curve associated with its bi-homogenization. Let $\sim$ be the smallest equivalence relation on $C$ such that $(a,b) \sim (c,d)$ whenever $a = c$ or $b=d$. The equivalence class of $(a,b)$ is called its orbit. We occasionally refer to it as the orbit of $a$. 
\end{Definition}

\section{Strategy}\label{sec:strat}

We reduce the non-linear problem of solving $r + qp = f - g$ to a linear one. The reduction is based on the computation of the poles of $f$ and $g$ and (upper bounds on) their multiplicities. This provides (bounds on) the denominators of $f$ and $g$ and the degrees of their numerators. 
It can be assumed that the denominator of $q$ is the least common multiple of the denominators of $f$, $g$ and $r$. The degree of its numerator can then be upper bounded from $r$, $p$ and what is known about $f$, $g$ and $q$. 
The numerators of $f$, $g$ and $q$ are found by making an ansatz for them, plugging it into $r + qp = f - g$, clearing denominators, and solving a system of linear equations for the unknowns of the ansatz that results from a comparison of coefficients.

To find the poles of $f$ and $g$ and their multiplicities, it is most natural to relate the poles of $r$ on the curve defined by $p$ to those of $f$ and $g$. If $(s_1,s_2)$ is a root of $p$ that is a pole of~$r$, then $s_1$ is a pole of $f$, or $s_2$ is a pole of $g$. Furthermore, if $(s_1,s_2)$ is not a pole of $r$, then $s_1$ is a pole of $f$ if and only if $s_2$ is one of $g$.
This suggests a procedure to compute the poles of $f$ and $g$. However, it raises several questions. 

Given a pole of $r$, at least one of its coordinates is a pole of $f$ or $g$, respectively. But it is not clear which of them. The propagation of poles of $f$ and $g$ starts at poles of $r$, and necessarily ends at poles of $r$. But it is not clear at which of them it ends. Unless $r$ defines a constant function, it has a pole. So the procedure always applies. But can all the poles of $f$ and $g$ be found in that way? If not, what are the remaining poles? We will address these questions in the next section. Though we can answer some of them, some questions remain open, see Section~\ref{sec:problems}.

When it comes to determining (upper bounds on) the multiplicities of the poles of $f$ and $g$, it is most natural to relate the local behavior of $r$ at its poles to that of $f$ and $g$. However, things are complicated. If $(s_1,s_2)$ is a root of $p$ and $s_1$ and $s_2$ are poles of $f$ and $g$, respectively, then the behavior of $r$ at $(s_1,s_2)$ does not necessarily reflect that of $f$ or $g$. The poles of $f$ and $g$ might affect each other in $f-g$. 

\section{Poles and multiplicities}\label{sec:alg}

We assume that $r\in\mathbb{K}[x,y]_p\setminus\langle p \rangle$, since if $r\notin\mathbb{K}[x,y]_p$, then $\mathrm{F}(r,p)=\emptyset$, and if $r\in\langle p \rangle$, then $\mathrm{F}(r,p) = \mathrm{F}(p)$. We furthermore assume that $\mathrm{F}(r,p)\neq \emptyset$ and denote by $q$, $f$ and $g$ elements of $\mathbb{K}[x,y]_p$, $\mathbb{K}(x)$ and $\mathbb{K}(y)$, respectively, such that 
\begin{equation}\label{eq:equation}
r + qp = f - g.
\end{equation} 

We now (try to) work out the strategy described in Section~\ref{sec:strat}. We begin with inspecting the poles of $r$ and relating them to the poles of $f$ and $g$. Our first observation is that, unless $r$ is a constant, it always has a pole. 
\begin{Proposition}\label{prop:pole}
Any non-constant rational function on an irreducible projective curve has a pole.
\end{Proposition}
\begin{proof}
We refer to~\cite[Ch~1, Sec~5, Cor~1]{shafarevich1994basic} for a proof.
\end{proof}


\begin{Proposition}\label{prop:propagation}
If $(s_1,s_2)$ is a pole of $r$ on the curve associated with $p$, then one of its coordinates is a pole of $f$ or $g$, respectively. If one of the coordinates of $(s_1,s_2)$ is a pole of $f$ or $g$, respectively, but the other is not, then its multiplicity can be read off from the local behavior of $r$ at $(s_1,s_2)$. If $(s_1,s_2)$ is not a pole of $r$, then $s_1$ is a pole of $f$ if and only if $s_2$ is one of $g$. If both coordinates are poles of $f$ and $g$, respectively, but $(s_1,s_2)$ is not a pole of $r$, then the ratio of their multiplicities can be read off from the local behavior of $r$ at $(s_1,s_2)$.
\end{Proposition}
\begin{proof}
Without loss of generality assume that $(\infty,\infty)$ is a point on the curve associated with $p$, and let $\varphi$ be an element of $\mathbb{K}\{\{x^{-1}\}\}$ such that $p(x,\varphi) = 0$ and $\deg \varphi > 0$. Plugging $\varphi$ into equation~\eqref{eq:equation} results in
\begin{equation}\label{eq:equation1.5}
r(x,\varphi) = f(x) - g(\varphi).
\end{equation}
Hence
\begin{equation}\label{eq:equation1}
\deg r(x,\varphi) \leq \max\{ \deg f, \deg g(\varphi) \},
\end{equation}
and
\begin{equation*}
\deg r(x,\varphi) \leq \mathrm{m}(\infty,f) \quad \text{or} \quad \deg r(x,\varphi)/ \deg \varphi \leq  \mathrm{m}(\infty,g).
\end{equation*}

If $(\infty,\infty)$ is a pole of $r$, then we can assume that $\deg r(x,\varphi) > 0$. Hence $\infty$ is a pole of $f$, or a pole of $g$. If $\infty$ is not a pole of $g$, for instance, then $\infty$ is a pole of $f$. Since $\deg f > 0$, but $\deg g \leq 0$, and hence $\deg g(\varphi) = \deg g \deg \varphi \leq 0$, we have $\deg r(x,\varphi) = \deg f$.

If $(\infty, \infty)$ is not a pole of $r$, but $\infty$ is a pole of $f$, for instance, then $\deg r(x, \varphi)\leq 0$ and $\deg f > 0$. So the leading terms of $f$ and $g(\varphi)$ in equation~\eqref{eq:equation1.5} cancel. In particular, $\deg f = \deg g \deg \varphi$. Hence $\deg g > 0$, and $\infty$ is a pole of $g$. Furthermore, 
\begin{equation*}
\mathrm{m}(\infty,g) = \mathrm{m}(\infty,f) /  \deg \varphi. 
\end{equation*}
\end{proof}

The proposition indicates a procedure for computing the poles of $f$ and $g$. It suggests to determine their poles by looking for the poles of $r$, and investigating how they propagate along their orbits. However, it raises several questions. If $(s_1,s_2)$ is a pole of $r$, which of its coordinates are poles of $f$ and $g$, respectively? Where does the propagation of poles stop? Can all poles of $f$ and $g$ be found among the orbits of the poles of $r$? And what are their multiplicities? 

We begin with addressing the question of how to derive upper bounds on the multiplicities of the poles of $f$ and $g$. Our first observation is that an upper bound for $\mathrm{m}(\infty,f)$ gives rise to an upper bound for $\mathrm{m}(\infty,g)$, and hence to upper bounds for $\mathrm{m}(s_1,f)$ and $\mathrm{m}(s_2,g)$ for any other element $(s_1,s_2)$ of the orbit of $(\infty,\infty)$.

If not otherwise stated, we assume throughout that $\varphi$ is as in the proof of Proposition~\ref{prop:propagation}, that is, an element of $\mathbb{K}\{\{x^{-1}\}\}$ such that $p(x,\varphi) = 0$ and $\deg \varphi > 0$,

\begin{Lemma}\label{lem:upperBound}

Assume that $\mathrm{m}(\infty,f) \leq b$. If $\deg r(x,\varphi) > b$, then $\mathrm{m}(\infty,g) = \deg r(x,\varphi) / \deg \varphi$, and if $\deg r(x,\varphi) \leq b$, then $\mathrm{m}(\infty,g) \leq 
b / \deg \varphi$.
\end{Lemma}
\begin{proof}
The statement is immediate from equation~\eqref{eq:equation1.5}. If $\deg r(x,\varphi) > b$, then the leading term of $r(x,\varphi)$ can only come from $g(\varphi)$. Hence $\deg r(x,\varphi) = \deg g(\varphi)$, and $\mathrm{m}(\infty,g) = \deg r(x,\varphi) / \deg \varphi$. If $\deg r(x,\varphi) \leq b$, then $\deg g(\varphi) \leq b$, and hence $\mathrm{m}(\infty,g) \leq b / \deg \varphi$.
\end{proof}

It can happen that $\mathrm{F}(r,p)$ consists of more than one element. To determine their poles and derive an upper bound for their multiplicities, we therefore have to make a choice for $f$ and $g$. When $\mathrm{F}(p)$ is trivial, there is not much to choose. Up to additive constants, there is only one element it consists of. The situation is clearly different when $\mathrm{F}(p)$ is non-trivial. Let $(f_0,g_0)$ be a generator of $\mathrm{F}(p)$ such that $\mathrm{m}(\infty,f_0) > 0$~\cite[Sec~5.1]{buchacher2024separating}. By subtracting a linear combination of powers of $(f_0,g_0)$ from $(f,g)$, we can assume that
\begin{equation*}
\mathrm{m}(\infty,f) = 0 \quad \text{or} \quad \mathrm{m}(\infty,f_0) \nmid \mathrm{m}(\infty,f). 
 \end{equation*}
 
The two cases have to be considered separately. We first assume that $\mathrm{F}(p)$ is non-trivial.

\subsection{$\mathrm{F}(p)\ncong \mathbb{K}$ }

Let $\varphi$ be as before, and define $\omega = (1, \deg \varphi)$. Note that $\mathrm{lp}_\omega(p)$ involves at least two terms, since $p(x,\varphi) = 0$, and hence $\mathrm{lp}_\omega(p)(x,\mathrm{lt}(\varphi)) = 0$.

If $r + qp = f - g$, and hence $qp = f - g - r$, then $q_d$ is a divisor of $f_dg_dr_d$. So if $\mathrm{lp}_\omega(q_d)$ has an irreducible factor that is not a single term, then it comes from $\mathrm{lp}_\omega(r_d)$, and hence from $r_d$. 
In particular, if $\mathrm{lp}_\omega(p)$ is a divisor of $\mathrm{lp}_\omega(q_d)$, then so it is of $\mathrm{lp}_\omega(r_d)$.
The next lemma shows that the latter can be avoided, hence guaranteeing that $\mathrm{lp}_\omega(p)/\mathrm{lp}_\omega(q_d)$ is not a single term. 

\begin{Lemma}
Let $A, B, Q, R \in\mathbb{K}[x,y]$ be such that $\gcd(B,R) = 1$. Then
\begin{equation*}
\frac{A}{Q B + R} = \frac{A}{R} \mod B.
\end{equation*}
\end{Lemma}
\begin{proof}
The statement is clearly true, since 
\begin{equation*}
\frac{A}{Q B + R} - \frac{A}{R} = \frac{AR - A(QB+R)}{(QB+R)R} = \frac{- AQB}{(QB+R)R},
\end{equation*}
and $B$ and $R$, and hence $B$ and $QB + R$ too, are relatively prime.
\end{proof}

Let us consider $\mathrm{lp}_\omega(r+qp)$. We claim that the following proposition holds.

\begin{Proposition}\label{prop:upperBound}
Let $\mathrm{F}(p)$ be non-trivial. If there is a pole of $f$ or $g$ among the orbit of $(\infty,\infty)$, then it can be assumed that $\mathrm{lp}_\omega(r + qp)$ is different from $\mathrm{lp}_\omega(qp)$. Hence $\omega(f)$, $\omega(g)\leq \omega(r)$, and so $\mathrm{m}(\infty,f) \leq \omega(r)$ and $\mathrm{m}(\infty,g) \leq \omega(r) / \deg \varphi$.
\end{Proposition}
\begin{proof}
If $\mathrm{lp}_\omega(r+qp)$ does not equal $\mathrm{lp}_\omega(qp)$, then it equals $\mathrm{lp}_\omega(r)$, $\mathrm{lp}_\omega(r) + \mathrm{lp}_\omega(qp)$, or something of weight smaller than $\omega(r)$. In any of these three cases, 
$\omega(r)$ is an upper bound for $\omega(f)$ and $\omega(g)$. If $\mathrm{lp}_\omega(r+qp)$ equals $\mathrm{lp}_\omega(qp)$, then, by multiplicativity of taking leading parts,
$\mathrm{lp}_\omega(q) \mathrm{lp}_\omega(p) = \mathrm{lp}_\omega(f - g)$.
Since $\mathrm{lp}_\omega(p)$ involves at least two terms, and because $\mathrm{lp}_\omega(p) / \mathrm{lp}_\omega(q_d)$ is not a single term, $\mathrm{lp}_\omega(q) \mathrm{lp}_\omega(p)$ involves at least two terms too. It follows that 
\begin{equation*}
\mathrm{lp}_\omega(q) \mathrm{lp}_\omega(p) = \mathrm{lp}_\omega(f) - \mathrm{lp}_\omega(g).
\end{equation*}
If 
\begin{equation*}
\mathrm{F}(\mathrm{lp}_\omega(p)) = \mathbb{K}((f_\omega, g_\omega)),
\end{equation*}
then there is a $k\in\mathbb{Z}$, which can be assumed to be positive, such that 
\begin{equation*}\label{eq:parameter}
\left( \mathrm{lp}_\omega(f), \mathrm{lp}_\omega(g)\right) = (f_\omega^k, g_\omega^k).
\end{equation*}
In particular, 
\begin{equation}\label{eq:mult}
(\mathrm{m}(\infty,f), \mathrm{m}(\infty,g)) = k\cdot (\deg f_\omega, \deg g_\omega).
\end{equation}
We claim that $\mathrm{lp}_\omega(r+qp)$ can be assumed to be different from $\mathrm{lp}_\omega(qp)$. If it were not, and if it were not for any element of the orbit of $(\infty,\infty)$, then each such element $(s_1,s_2)$ gives rise to
a $1$-parameter family of candidates for $(\mathrm{m}(s_1,f),\mathrm{m}(s_2,g))$ as for $\mathrm{m}(\infty,f)$ and $\mathrm{m}(\infty,g)$ above. These $1$-parameter families merge to a single $1$-parameter family, and 
the construction of $(f_0,g_0)$ in~\cite[Sec~5]{buchacher2024separating} shows that there is an integer $k$ such that for any such element $(s_1,s_2)$ 
we have
\begin{equation*}
(\mathrm{m}(s_1,f), \mathrm{m}(s_2,g)) = k\cdot (\mathrm{m}(s_1,f_0), \mathrm{m}(s_2,g_0)). 
\end{equation*} 
If $k>0$, then $\mathrm{m}(\infty,f)$ is a non-zero multiple of $\mathrm{m}(\infty,f_0)$, contradicting the choice of $f$ and $g$. If $k=0$, then there are no poles of $f$ and $g$ found among the (coordinates of the) elements of the orbit of $(\infty,\infty)$, contradicting our assumptions.
\end{proof}

Let us point out that there is some uncertainty in the phrase ``it can be assumed''. It is meant that there is a transformation $T$ of $\mathbb{P}^1(\mathbb{K})\times \mathbb{P}^1(\mathbb{K})$ that maps an element of the orbit to $(\infty, \infty)$ such that the statement holds for $r\circ T^{-1}$ and (the numerator of) $p\circ T^{-1}$. A priori we do not know which transformation it is. However, the orbit of $(\infty,\infty)$ is finite, see~\cite[Thm~1]{fried1978poncelet} or~\cite[Sec~5.1]{buchacher2024separating}. So there are only finitely many points to consider. By Lemma~\ref{lem:upperBound}, each of them gives rise to possible upper bounds for the multiplicities of the poles of $f$ and $g$ found among this orbit. Some may not be, but those which are maximal are indeed upper bounds. The next algorithm summarizes these observations.

\begin{Algorithm}\label{alg:mult1}
 Input: an irreducible polynomial $p\in\mathbb{K}[x,y] \setminus \left( \mathbb{K}[x] \cup \mathbb{K}[y] \right)$ such that $\mathrm{F}(p) \ncong \mathbb{K}$, an element $r \in\mathbb{K}[x,y]_p\setminus \langle p\rangle$ such that $\mathrm{F}(r,p) \neq \emptyset$, and an orbit~$\mathcal{O}$.\\
  Output: a function $\mathfrak{m}: \mathcal{O} \rightarrow \mathbb{Q}^2$ such that for each $(s_1,s_2)\in\mathcal{O}$
  \begin{equation*}
  \mathfrak{m}(s_1,s_2) \geq \left( \mathrm{m}(s_1,f), \mathrm{m}(s_2,g) \right)
  \end{equation*}
  for some $(f,g)\in\mathrm{F}(r,p)$ that does not depend on $\mathcal{O}$.
  \step 10 let $\mathfrak{m}$ be the function that is constant on $\mathcal{O}$ and equal to $(-\infty,-\infty)$. 
  \step 20 for each $(s_1,s_2)\in\mathcal{O}$, do:
  \step 31 determine a transformation $T$ of $\mathbb{P}^1(\mathbb{K})\times\mathbb{P}^1(\mathbb{K})$ that maps $(s_1,s_2)$ to $(\infty,\infty)$, and replace $r$ by $r\circ T^{-1}$ and $p$ by the numerator of $p\circ T^{-1}$.  
  \step 41 set $\omega = (1,\deg \varphi)$, where $\varphi\in\mathbb{K}\{\{x^{-1}\}\}$ is such that $p(x,\varphi) = 0$ and $\deg \varphi > 0$.
  \step 51 replace the denominator of $r$ by the remainder of $r_d$ when reduced by $p$ with respect to $\omega$-weighted degree lexicographic order with $x > y$.
  \step 61 for each element $(\tilde{s}_1,\tilde{s}_2)$ of $\mathcal{O}$, do:
  \step 72 propagate $(\omega(r),\omega(r) / \deg \varphi)$ from $(\infty,\infty)$ to $T(\tilde{s}_1,\tilde{s}_2)$ to some (potential upper bound)  $(b_1,b_2)\in\mathbb{Q}^2$, and set 
   \begin{equation*}
   \mathfrak{m}(\tilde{s}_1,\tilde{s}_2) = \max\{\mathfrak{m}(\tilde{s}_1,\tilde{s}_2), (b_1,b_2)\},
   \end{equation*}
    where $\max$ is  understood to be applied component-wise.
 \step 80 return $\mathfrak{m}$.
\end{Algorithm}

We next draw our attention to the question whether all poles of $f$ and $g$ can be found among the orbits of the poles of $r$. We claim that the following proposition holds.

\begin{Proposition}\label{prop:poles1}
If $\mathrm{F}(p)$ is non-trivial, then the poles of $f$ and $g$ are among the (coordinates of the elements of the) orbits of the poles and roots of $r$.
\end{Proposition}
\begin{proof}
Assume there is no pole of $r$ among the orbit of $(\infty,\infty)$ though the coordinates of all its elements are poles of $f$ and $g$, respectively. We have seen in the proof of Proposition~\ref{prop:upperBound} that if $\mathrm{F}(p)$ is non-trivial, then it can be assumed that $\mathrm{lp}_\omega(r + qp) \neq \mathrm{lp}_\omega(qp)$, and hence $\omega(r) \geq \omega(f - g)$. Since $\infty$ is a pole of $f$ and $g$, and therefore $\omega(f - g) > 0$, it follows that $\omega(r) > 0$ too. Furthermore, $(\infty,\infty)$ is not a pole of $r$, so $\deg r(x,\varphi) \leq 0$. Hence the inequality in  $\deg r(x,\varphi) \leq \omega(r)$ is strict. Consequently, $\mathrm{lp}_\omega(r)(x,\mathrm{lt}(\varphi)) = 0$. In particular, $\mathrm{lp}_\omega(r)$ involves at least two terms. So $(\infty,\infty)$ is a root of $r$, and therefore a common root of $r$ and $p$. 
\end{proof}

Based on Proposition~\ref{prop:poles1} and Proposition~\ref{prop:upperBound} we can now formulate the following algorithm for computing (an element of) $\mathrm{F}(r,p)$ when $\mathrm{F}(p)$ is non-trivial.

\begin{Algorithm}\label{alg:1}
  Input: an irreducible polynomial $p\in\mathbb{K}[x,y] \setminus \left( \mathbb{K}[x] \cup \mathbb{K}[y] \right)$ such that $\mathrm{F}(p) \ncong \mathbb{K}$ and a rational function $r\in\mathbb{K}[x,y]_p\setminus \langle p \rangle$.\\
  Output: an element $(f,g)$ of $\mathrm{F}(r,p)$, if it is non-empty, otherwise $\emptyset$.
  \step 10 compute the orbits of the poles and roots of $r$ on $\{p = 0\}$.
  \step 20 for each such orbit, do:
  \step 31 compute an upper bound $\mathfrak{m}$ for the multiplicities of the poles found among the orbit for a (potential) element $(f,g)$ of $\mathrm{F}(r,p)$ as described by Algorithm~\ref{alg:mult1}. 
  \step 40 make an ansatz for $f$ and $g$ with poles $s_1$ and $s_2$, respectively, of multiplicity $\mathfrak{m}(s_1,s_2)$, where $(s_1,s_2)$ ranges over the elements of the orbits of the roots and poles of $r$, and make an ansatz for $q$ with $q_d = f_dg_dr_d$ and $\deg q_n = \max\{ \deg f_n g_d r_d, \deg g_nf_d r_d, \deg r_nf_dg_d \} - \deg p$. 
  \step 50 reduce $(f-g-r)q_d$ by $p$ and equate the coefficients of the remainder to zero. 
  \step 60 solve the resulting linear system for the unknowns of the ansatz.
  \step 70 if there is no solution, return $\emptyset$.
  \step 80 otherwise, return a pair $(f,g)$ corresponding to one of its solutions.
\end{Algorithm}


The following example illustrates Algorithm~\ref{alg:1} .

\begin{Example}
Let us consider 
\begin{equation*}
 r = xy \quad \text{and} \quad p = xy-x-y-x^2y^2.
 \end{equation*}
It was shown in~\cite[Example~4]{buchacher2024separating}, and can easily be verified, that 
\begin{equation*}
\mathrm{F}(p) = \mathbb{K}\left(\left(  \frac{1 - x - x^3}{x^2}, \frac{1 - y - y^3}{y^2}\right)\right).
\end{equation*}

To determine the poles of an element $(f,g)$ of $\mathrm{F}(r,p)$, we compute the orbits of the roots and poles of $r$. There are no finite poles, since $r$ is a polynomial. Hence they are among $(\infty,0)$ and $(0,\infty)$. By Proposition~\ref{prop:pole}, and due to symmetry, both of them are poles of $r$. The only root of $r$ is $(0,0)$. The orbits of these points are all the same. They consist of 
\begin{equation*}
(0,0), \quad (0,\infty), \quad \text{and} \quad (\infty,0).
\end{equation*}
The poles of $f$ and $g$ are therefore among $\{0,\infty\}$. 

We next derive upper bounds on their multiplicities. Let 
\begin{equation*}
\varphi = -x + \dots 
\end{equation*}
be the series root of $p$ in $\mathbb{K}\{\{x\}\}$ giving rise to a local parametrization of the curve associated with $p$ at $(0,0)$. Furthermore, define $\omega = -(1,\val \varphi)$. If $\mathrm{lp}_\omega(r+ qp) \neq \mathrm{lp}_\omega(qp)$, then $\omega(f)$, $\omega(g) \leq \omega(r)$, and hence
\begin{equation*}
\mathrm{m}(0,f),\hspace{2pt} \mathrm{m}(0,g) \leq -2.
\end{equation*} 
We do a similar analysis for $(\infty, 0)$ with $\omega = (1, -1/2) $ coming from $\psi = \mathrm{i} x^{-1/2} + \dots$,
one of the series roots of $p$ in $\mathbb{K}\{\{x^{-1}\}\}$, describing one of the branches at $(\infty,0)$. If $\mathrm{lp}_\omega(r+ qp) \neq \mathrm{lp}_\omega(qp)$, then $\omega(f)$, $\omega(g) \leq \omega(r)$, and hence
\begin{equation*}
\mathrm{m}(\infty,f) \leq 1/2 \quad \text{and} \quad \mathrm{m}(0,g) \leq 1.
\end{equation*} 
By symmetry, we have 
\begin{equation*}
\mathrm{m}(0,f) \leq 1 \quad \text{and} \quad \mathrm{m}(\infty,g) \leq 1/2. 
\end{equation*}

We investigate how these potential upper bounds propagate. We exemplify the argument for $\mathrm{m}(0,g)$. We only have to investigate how $\mathrm{m}(0,f) \leq 1$ propagates to an upper bound for $\mathrm{m}(0,g)$. If $\mathrm{m}(0,f) \leq 1$, then $\val r(x,\varphi) > \min\{\val f, \val g(\varphi)\}$, and therefore $\val f = \val g(\varphi)$. It follows that $\mathrm{m}(0,g) \leq 1$. Comparing with the other potential bounds, we find that $\mathrm{m}(0,g) = 1$. Similarly, one can show that $\mathrm{m}(0,f) = 1$, $\mathrm{m}(\infty,f) = 0$ and $\mathrm{m}(\infty,g) = 0$.
%
%
%
%

If $(f,g)$ is an element of $\mathrm{F}(r,p)$, then the poles of $f$ and $g$ are $0$, and their multiplicities are $1$. 
Indeed, making the ansatz
\begin{equation*}
f = \frac{f_0 + f_1 x}{x}, \quad g = \frac{g_0 + g_1 y}{y} \quad \text{and} \quad q = \frac{q_{00}}{xy},
\end{equation*}
clearing denominators in $r + qp = f +g$ and comparing coefficients, results in a system of linear equations for the undetermined coefficients. Its solutions correspond to 
\begin{equation*}
f = \frac{ux-1}{x}, \quad g = \frac{(1-u)y-1}{y} \quad \text{and} \quad q = \frac{1}{xy}
\end{equation*}
for $u\in\mathbb{K}$.
In particular, we find that
\begin{equation*}
r + \frac{1}{xy} p = \frac{x-1}{x} - \frac{1}{y}.
\end{equation*}
\end{Example}

\subsection{$\mathrm{F}(p) \cong \mathbb{K}$}

We assumed that $\mathrm{F}(p)$ was non-trivial. We now turn to the other case. Let $\mathrm{F}(r,p)$ be non-empty, and let $(f,g)$ be an element thereof. We begin with the following observation.

\begin{Proposition}\label{prop:infinite}
Any infinite orbit among the coordinates of whose elements are poles of $f$ or $g$ contains a pole of $r$. 
\end{Proposition}
\begin{proof}
Assume the orbit did not contain a pole of $r$, so that for each of its elements one of its coordinates is a pole of $f$ or $g$, respectively, if and only if the other is. By assumption, there is an element one of whose coordinates is a pole of $f$ or $g$, respectively. Hence so is the other coordinate, and hence so are the coordinates of any other element of the orbit. However, the orbit is infinite, while $f$ and $g$ have only finitely many poles. A contradiction.
\end{proof}

The computation of the multiplicities of the poles of an element of $\mathrm{F}(r,p)$ that are found among an infinite orbit relies on the existence of a pole of $r$ one of whose coordinates is not a pole of $f$ or $g$, respectively. 

\begin{Proposition}\label{prop:multiplicities}
Assume the orbit of a pole of $r$ is infinite. Then there is a pole of $r$ one of whose coordinates is not a pole of $f$ or $g$, respectively.
\end{Proposition}
\begin{proof}
If there were no such pole, then a coordinate of an element of the orbit is a pole of $f$ or $g$, respectively, if and only if the other is. Hence we can argue as in the proof of Proposition~\ref{prop:infinite}.
\end{proof}

Whether a given pole of $r$ is a pole as in Proposition~\ref{prop:multiplicities} can sometimes be deduced from the infiniteness of (parts of) its orbit and the location of the (other) poles of $r$. In the following we make this more precise using what we call the graph of an orbit.

\begin{Definition}
The \textbf{graph of an orbit} is the graph whose vertices are the elements of the orbit and whose edges are those pairs of vertices which share a common coordinate. A path is a sequence $v_0, v_1, v_2,\dots$ of vertices for which any two consecutive vertices $v_i$, $v_{i+1}$ form an edge. The vertex $v_0$ is said to be its starting point. If the path is finite and equal to $v_0,\dots, v_n$, then $v_n$ is called its ending point. If $n>0$ and if $v_0 = v_n$, then the path is said to be a cycle.
\end{Definition}

The following proposition is immediate.

\begin{Proposition}\label{prop:specialPole}
Let $(s_1,s_2)$ be a pole of $r$, and let $v_0, v_1,\dots$ be an infinite path in the graph of its orbit that starts at $(s_1,s_2)$ and does not contain a cycle. If the path does not contain another pole of $r$, then one of the coordinates of $(s_1,s_2)$ cannot be a pole of $f$ or $g$, respectively. If the first coordinate of $v_1$ equals $s_1$, then $s_1$ is not a pole of $f$. If its second coordinate equals $s_2$, then $s_2$ is not pole of $g$. 
\end{Proposition}

Not necessarily every pole of $r$ one of whose coordinates is not a pole of $f$ or $g$, respectively, can be identified using Proposition~\ref{prop:specialPole}. However, there is always a pole of $r$ to which it applies.

\begin{Proposition}
Assume that the orbit of a pole of $r$ is infinite. Then there is an infinite cycle-free path in the graph of its orbit that starts at some pole of $r$ and does not contain any other pole of $r$.
\end{Proposition}
\begin{proof}
The orbit is infinite, so it contains an infinite cycle-free path, and of course we can assume it contains a pole of $r$. 
Let $v_0, v_1,\dots$ be its vertices, and let $p_1, \dots,p_k$ be the poles of $r$ encountered thereon as one goes along it. If $v_n = p_k$, then $v_n, v_{n+1},\dots$ are the vertices of a path with the required property.
\end{proof}

To formulate a semi-algorithm, it will be convenient to have the following definition.

\begin{Definition}  
Let $\mathcal{O}$ be an orbit, and let $p_1,\dots, p_k$ be elements thereof. We denote by $\mathcal{O}(p_1,\dots,p_k)$ the set of elements $v$ of $\mathcal{O}$ for which there is no infinite cycle-free path that contains $v$ but none of $p_1,\dots,p_k$.
\end{Definition}

\begin{Proposition}
Let $\mathcal{O}$ be an orbit, and let $p_1,\dots, p_k$ be the poles of $r$ thereon which fulfill the assumptions of Proposition~\ref{prop:specialPole}. Then $\mathcal{O}(p_1,\dots,p_k)$ is finite, and the poles of $f$ and $g$ are among the coordinates of its elements. Furthermore, the elements of the boundary of $\mathcal{O}(p_1,\dots,p_k)$ are $p_1,\dots, p_k$.
\end{Proposition}
\begin{proof}
If $\mathcal{O}$ is finite, then $\mathcal{O}(p_1,\dots,p_k)$ is all of $\mathcal{O}$, and the proposition is clearly true. So let us assume that $\mathcal{O}$ is infinite. Assume there is a vertex of $\mathcal{O}$ one of whose coordinates is a pole of $f$ or $g$, respectively, but which is not an element of $\mathcal{O}(p_1,\dots,p_k)$. Then it is the vertex of an infinite cycle-free path none of $p_1,\dots,p_k$ is an vertex of. Hence the coordinates of all its vertices are poles of $f$ and $g$, respectively. A contradiction. So all poles of $f$ and $g$ are among the coordinates of the elements of $\mathcal{O}(p_1,\dots,p_k)$. To see that $\mathcal{O}(p_1,\dots,p_k)$ is finite, note that the connected components of the subgraph induced by $\mathcal{O}(p_1,\dots,p_k)$ are finite and each connected component contains a pole of $r$. Since $r$ has only finitely many poles, $\mathcal{O}(p_1,\dots,p_k)$ is finite too. The last part of the statement is obvious.
\end{proof}




When $\mathrm{F}(p)$ is non-trivial, then every orbit is finite, and there is a uniform bound on their size. The situation is quite different when $\mathrm{F}(p)$ is trivial, in which case the orbit of a generic point is infinite~\cite{fried1978poncelet}. The following conjecture makes this more precise. 

\begin{Conjecture}\label{prop:finite}
If $\mathrm{F}(p)$ is trivial, then there are only finitely many finite orbits.
\end{Conjecture}

Proposition~\ref{prop:finite} and Proposition~\ref{prop:infinite} imply the following corollary.

\begin{Corollary}
If $\mathrm{F}(p)$ is trivial, then the poles of $f$ and $g$ are among (the coordinates of the elements of) the orbits of the poles of $r$ and the finite orbits. 
\end{Corollary}

It is not difficult to formulate an algorithm that computes all orbits whose size does not exceed a given number $N$. Let $\pi_1$ and $\pi_2$ be the projections that map a point on the curve associated with $p$ on its first and second coordinate, respectively. Furthermore, let $\mathrm{res}_y(p_1,p_2)$ denote the resultant of $p_1$, $p_2 \in\mathbb{K}[x,y]$ with respect to $y$, and let $\mathrm{sfp}(p_1)$ be the square-free part of $p_1$. Define
\begin{equation*}
Q_0(x,y) = p(x,y),
\end{equation*}
and 
\begin{equation*}
Q_{n+1}(x,y) = \mathrm{sfp}\left(\mathrm{res}_z\left( Q_n(x,z), Q_n(y,z) \right)\right), \quad n\geq 0.
\end{equation*}

\begin{Lemma}\label{lemma:finiteOrbits}
Let $n > 0$, and assume that $x_0$, $y_0\in\mathbb{K}$ of which $x_0$ is not a root of $\mathrm{lc}_y(Q_k)$ for $k < n$, and that $\infty$ is not an element of the orbit of $x_0$. Then
\begin{equation*}
Q_n(x_0,y_0) = 0 \quad \Longleftrightarrow \quad y_0\in\left(\pi_1\circ\pi_2^{-1}\circ\pi_2\circ\pi_1^{-1}\right)^{2^{n-1}}(x_0).
\end{equation*}
\end{Lemma}

\begin{proof}
By~\cite[Chap 3, Sec~6, Prop~6]{cox2015ideals}, $y_0$ is a root of $Q_n(x_0,z)$ if and only if it is a root of $\mathrm{res}_z(Q_{n-1}(x_0,z),Q_{n-1}(y_0,z))$, and hence a common root of $Q_{n-1}(x_0,z)$ and $Q_{n-1}(y_0,z)$. 
If $n = 1$, this means that $y_0$ is a common root of $p(x_0,z)$ and $p(y_0,z)$. Clearly, this is the case if and only if  $y_0\in\pi_1(\pi_2^{-1}(\pi_2(\pi_1^{-1}(x_0))))$. 
So let $n > 1$, and assume the statement holds for $n -1$. Then $y_0$ is a root of $Q_n(x_0,z)$ if and only if
$\left(\pi_1\circ\pi_2^{-1}\circ\pi_2\circ\pi_1^{-1}\right)^{2^{n-1}}(x_0)$ and $\left(\pi_1\circ\pi_2^{-1}\circ\pi_2\circ\pi_1^{-1}\right)^{2^{n-1}}(y_0)$ have non-trivial intersection. Obviously, this is the case if and only if $y_0 \in \left(\pi_1\circ\pi_2^{-1}\circ\pi_2\circ\pi_1^{-1}\right)^{2^n}(x_0)$.
\end{proof}

To compute all orbits of size at most $N$, one first determines some $Q_n$ whose degree with respect to $y$ is larger than $N$. If there were none, then there would be some $n$ for which $Q_n = Q_{n+1}$, and so every orbit would be finite. Let $(x_0,y_0)$ be a root of $p$, and assume that $x_0$ is an element of $\mathbb{K}$ that is not a root of $\mathrm{lc}_y(Q_{k})$ for any $k < n$. If $x_0$ is not a root of $\mathrm{res}_y\left(Q_n, \frac{\partial Q_n}{\partial y}\right)$, then $Q_n(x_0,z)$ does not have any multiple roots, and the size of the orbit of $x_0$ exceeds $N$. If $x_0$ is a root, then the orbit of $(x_0,y_0)$ still might have more than $N$ elements, and whether it does or not can be verified by computing at most $N+1$ elements. The same is true, if $x_0$ is a root of one of the $\mathrm{lc}_y(Q_{k})$'s, or if it equals $\infty$.

In order to formulate a semi-algorithm that determines an element $(f,g)$ of $\mathrm{F}(r,p)$, it remains to clarify how to compute an upper bound for the multiplicities of the poles of $f$ and $g$ that can be found among the coordinates of the elements of a finite orbit. 
We cannot hope for an analogous result as Proposition~\ref{prop:upperBound} as it relies on $\mathrm{F}(p)$ being non-trivial. When $\mathrm{F}(p)$ is trivial, but $\mathrm{F}(\mathrm{lp}_\omega(p))$ is not, 
it may be that $\mathrm{lp}_\omega(r + qp) = \mathrm{lp}_\omega(qp)$, and hence $\mathrm{lp}_\omega(f - g) = f_\omega^k - g_\omega^k$, for a generator $(f_\omega,g_\omega)$ of $\mathrm{F}(\mathrm{lp}_\omega(p))$, and some integer $k$. Hence we need to wonder about the value of $k$. We believe that if $k$ is big, then so is the remainder of (the numerator of) $f-g$ when reduced by $p$. 
We have the following conjecture.

\begin{Conjecture}\label{conj:trivial} 
If $\mathrm{F}(p)$ is trivial, then the integer $k$ in equation~\eqref{eq:mult} is minimal such that 
\begin{equation*}
\omega(r) < k \cdot \omega(f_\omega).
\end{equation*} 
\end{Conjecture}

The next example illustrates this conjecture.

\begin{Example}
The equation 
\begin{equation*}
r + qp = f - g
\end{equation*}
clearly has a solution for 
\begin{equation*}
r = -3xy \quad \text{and} \quad p = x^3+3xy-y^3.
\end{equation*}
Still, let us discuss this trivial example, and let us pretend we are interested in polynomial solutions only. 
The polynomial $p$ does not have a separated multiple. However, its leading part with respect to $\omega = (1,1)$ does. We have
\begin{equation*}
\mathrm{F}(p) = \mathbb{K}((x^3, y^3)).
\end{equation*}
The smallest integer $k$ for which $\omega(r) < k \cdot \omega(x^3)$ is $k = 1$. This is consistent with this example, since 
\begin{equation*}
r + p = x^3 - y^3.
\end{equation*}
The example also supports the underlying intuition that if the leading part of $f-g$ is big, then so is its remainder. The term of the remainder of $x^{3k} - y^{3k}$ which is maximal with respect to lexicographic order with $y > x$ when reduced by $p$ with respect to $\omega$-weighted degree lexicographic order with $x > y$ is $-3kxy^{3(k-1)+1}$. To see this, observe that 
\begin{equation*}
x^{3k} - y^{3k} = (x^3-y^3) \sum_{i = 1}^k \binom{k}{i} (x^3-y^3)^{i-1}y^{k-i},
\end{equation*}
the remainder of $x^3-y^3$ is $-3xy$, and the term of the remainder of $\sum_{i = 1}^k (x^3-y^3)^{i-1}y^{k-i}$ which is maximal with respect to weighted degree lexicographic order with $x > y$ equals $ky^{k-1}$. Both are maximal with respect to lexicographic order with $y > x$. Hence the claim follows. Since a term only in $y$ is already reduced, and because the maximal term in the remainder of a term $x^{3k + l}$ with $0 \leq l < 3$ is $x^l y^{3k}$, it follows that $-3kxy^{3(k-1)+1}$ is also the maximal term in the remainder of any polynomial of the form $f - g$ with $\mathrm{lp}_\omega(f-g) = x^{3k} - y^{3k}$. 
\end{Example}

\begin{Remark}\label{rem:conj}
Note that if $\mathrm{F}(\mathrm{lp}_\omega(p))$ is trivial, then $\mathrm{lp}_\omega(r+qp)$ is different from $\mathrm{lp}_\omega(qp)$, and $\omega(r)$ provides an upper bound for $\omega(f)$ and $\omega(g)$. So there is no need to worry about Conjecture~\ref{conj:trivial} in this situation.
\end{Remark}

We summarize our conjectural algorithm.

\begin{Algorithm}\label{alg:2}
  Input: an irreducible polynomial $p\in\mathbb{K}[x,y] \setminus \left( \mathbb{K}[x] \cup \mathbb{K}[y] \right)$ such that $\mathrm{F}(p) \cong \mathbb{K}$, and an element $r \in\mathbb{K}[x,y]_p\setminus \langle p\rangle$ such that $\mathrm{F}(r,p) \neq \emptyset$.\\
  Output: an element $(f,g)$ of $\mathrm{F}(r,p)$.
  \step 10 compute the finite orbits and the poles of $r$. 
  \step 20 for each infinite orbit $\mathcal{O}$ that contains a pole of $r$, do: 
  \step 31 identify the poles $p_1,\dots, p_k$ of $r$ in $\mathcal{O}$ only one of whose components is a pole of $f$ or $g$, respectively, for which the assumptions of Proposition~\ref{prop:specialPole} are satisfied, and compute $\mathcal{O}(p_1,\dots,p_k)$.
  \step 41 determine an upper bound on the multiplicities of the poles of $f$ and $g$ that are among the coordinates of the elements of $\mathcal{O}(p_1,\dots,p_k)$ based on Proposition~\ref{prop:propagation} and Lemma~\ref{lem:upperBound}.
  \step 50 for each finite orbit $\mathcal{O}$, do:
  \step 61 compute an upper bound on the multiplicities of the poles of $f$ and $g$ found among the coordinates of the elements of $\mathcal{O}$ based on Remark~\ref{rem:conj}, Conjecture~\ref{conj:trivial} and Lemma~\ref{lem:upperBound}.
  \step 70 make an ansatz for $f $ and $g$ based on the (estimate of the set of) poles and the upper bounds on their multiplicities determined in steps~$2-6$, and for $q$ with $q_d = f_dg_dr_d$ and $\deg q_n = \max\{ \deg f_n g_d r_d, \deg g_nf_d r_d, \deg r_nf_dg_d \} - \deg p$. 
  \step 80 reduce $(f-g-r)q_d$ by $p$ and equate the coefficients of the remainder to zero. 
  \step 90 solve the resulting linear system for the unknowns of the ansatz.
  \step {10}0 return a pair $(f,g)$ corresponding to one of its solutions.
\end{Algorithm}

The algorithm is formulated in a somewhat imprecise way, since it is supposed to deal with infinite objects, and we have not explained how this can be done. Step~$1$ asks to compute the finite orbits, the for-loop in step~2 loops over all infinite orbits that contain a pole of $r$, and step~3 is based on Proposition~\ref{prop:specialPole} whose application relies on the existence of certain infinite paths. However, neither do we know how big a finite orbit can get, nor how to decide the infiniteness of an orbit, or the existence of certain infinite paths. Nevertheless, in practice one can circumvent these problems. An orbit can be considered to be infinite if its cardinality exceeds a given size. Similarly, an infinite path having certain properties can be considered to exist, if there is such a path that may be finite but is sufficiently large. One can formulate a variant of Algorithm~\ref{alg:2} that takes an additional parameter which specifies when a possibly finite object is considered to be infinite. If $\mathrm{F}(r,p)$ is non-empty, and the parameter is chosen large enough, then it is guaranteed to output an element thereof. But if $\mathrm{F}(r,p)$ is empty, then the algorithm will not find any, and it is not clear whether it is because the parameter is not suitably chosen or $\mathrm{F}(r,p)$ is indeed empty.   


We close this section with another example.

\begin{Example}
Let 
\begin{equation*}
r = x y \quad \text{and} \quad p = x y-xy^2-x^2y-x^2-x-y^2.
\end{equation*}
The poles of $r$ are
\begin{equation*}
 (-1,\infty), \quad (\infty,\infty) \quad \text{and} \quad (\infty, -1).
\end{equation*}
The orbits of these points are all the same. They consist of 
\begin{equation*}
(0,0), \quad (-1,0), \quad (-1,\infty), \quad (\infty,\infty) \quad \text{and} \quad (\infty,-1)
\end{equation*}
and (conjecturally) infinitely many other points. Since $(\infty,-1)$ is a pole of $r$, one of its coordinates is a pole of $f$ or $g$, respectively. Let us assume that the orbit is infinite, in which case $-1$ cannot be a pole of $g$, as otherwise $f$ and $g$ had infinitely many poles. It then follows that $\infty$ is a pole of $f$. 

To determine $\mathrm{m}(\infty,f)$, we compute $\varphi = -1 + \dots$, the series root of $p$ in $\mathbb{K}\{\{x^{-1}\}\}$ that describes the local behavior of $\{p = 0\}$ at $(\infty,-1)$. Since $\deg r(x,\varphi) = 1$, we have $\mathrm{m}(\infty,f) = 1$.

Let us see how this multiplicity propagates to upper bounds for the potential other poles of $f$ and $g$. Let now $\varphi = -x + \dots$ be the series root of $p$ in $\mathbb{K}\{\{x^{-1}\}\}$ describing the local behavior of $\{p = 0\}$ at $(\infty,\infty)$. Since $\deg r(x,\varphi) = 2 > \deg f$, we not just get an upper bound for $\mathrm{m}(\infty,g)$, but we find that $\mathrm{m}(\infty,g) = 2$.

We continue in a similar manner with $(-1,\infty)$. We move $(-1,\infty)$ to $(0,\infty)$ by considering $r(x-1,y)$ and $p(x-1,y)$. The series root in $\mathbb{K}\{\{y^{-1}\}\}$ that describes the local behavior of $\{p(x-1,y) = 0\}$ at $(0,\infty)$ is $\varphi = -2/y + \dots$. Since $\deg r(\varphi,y) = 1 < \deg g$, it follows that $\deg f(\varphi -1 ) = \deg g$. Hence $
\mathrm{m}(-1,f) = 2$. 

Since $r(-1,0)$ and $r(0,0)$ are finite, it follows that $0$ is a pole of both $f$ and $g$. It remains to determine $\mathrm{m}(0,g)$ and $\mathrm{m}(0,f)$. Again, we move $(-1,0)$ to $(0,0)$ by considering $r(x-1,y)$ and $p(x-1,y)$. The series root in $\mathbb{K}\{\{y\}\}$ describing the local behavior at $(0,0)$ is $\varphi = 2y + \dots$. Since $\val r(\varphi-1,y) > -2$, it follows that $\val f(\varphi-1) = \val g$. Hence $\mathrm{m}(0,g) = 2$.

The series roots in $\mathbb{K}\{\{x\}\}$ that describe the local behavior of $\{p = 0\}$ at $(0,0)$ are $\varphi = \pm \mathrm{i} x^{1/2} \dots$. Since $r(0,0) < 0$, we have $\val f = \val g(\varphi)$. It follows that $\mathrm{m}(0,f) = 1$.

Making an ansatz for $f$, $g$ and $q$, and solving a system of linear equations for the unknowns of the ansatz, we find that
\begin{equation*}
r + \frac{1+x-2x y+x y^2+x^2y^2}{x(1+x)^2y^2}p = \frac{-1-2x-4x^2-x^4}{x(1+x)^2} + \frac{-1+y+2y^3-y^4}{y^2}.
\end{equation*}
\end{Example}


\section{Open problems and conjectures}\label{sec:problems}
The (conjectural) procedures for computing an element of $\mathrm{F}(r,p)$ are not algorithms as they do not terminate on any input. They terminate whenever $\mathrm{F}(r,p)$ is non-empty or $\mathrm{F}(p)$ is non-trivial. However, they might not terminate if this is not the case. The question how they could be turned into algorithms gives rise to several problems in algebraic dynamics. 

The construction of a generator of $\mathrm{F}(p)$ in~\cite[Sec~5.1, Sec~6]{buchacher2024separating} has similarities with the computation of an element of $\mathrm{F}(r,p)$. Both are based on the computation of their poles and the corresponding multiplicities. It was explained in~\cite[Sec~5.1]{buchacher2024separating} that the poles of a generator of $\mathrm{F}(p)$ appear in pairs. They correspond to points on the curve associated with $p$ and constitute the orbit of $\infty$. If $\mathrm{F}(p)$ is non-trivial, then the orbit of $\infty$ is finite. However, it might be infinite when $\mathrm{F}(p)$ is trivial, in which case
 the semi-algorithm does not terminate. This raises the following problem. 

\begin{Problem}
Given a root of $p$, decide whether its orbit is finite.
\end{Problem}

If $\mathrm{F}(p)$ is trivial, then the poles of an element of $\mathrm{F}(r,p)$ are found among the orbits of the poles of $r$ and the finite orbits. We believe there are only finitely many finite orbits.

\setcounter{Conjecture}{1}

\begin{Conjecture}
If $\mathrm{F}(p)$ is trivial, then there are only finitely many finite orbits.
\end{Conjecture} 

We saw how to determine finite orbits of a given size. However, we do not know how big a finite orbit can get. 

\begin{Problem}
Given $p$, determine an upper bound on the size of its finite orbits.
\end{Problem}

The computation of the multiplicities of the poles of an element of $\mathrm{F}(r,p)$ found among an orbit that is infinite relies on the identification of specific poles and a solution of the following problem.

\begin{Problem}
Let $(s_1,s_2)$ and $S$ be a root and a finite set of roots of $p$, respectively. Decide whether there is an infinite cycle-free path $v_0,v_1,\dots$ in the graph of the orbit of $(s_1,s_2)$ such that $v_0$ equals $(s_1,s_2)$, the first coordinate of $v_1$ equals $s_1$, and no element of $S$ is among its vertices.
\end{Problem}

We point out that these problems can sometimes be circumvented. As already observed in~\cite[Cor~2]{buchacher2024separating}, the triviality of $\mathrm{F}(p)$ can often be read off from the Newton polygon of $p$. Similar observations can be made for $\mathrm{F}(r,p)$ as the following examples indicates.


\begin{Example}
Consider $r = xy$ and $p = 1 + x^3+x^2y^2+y^3$. We show that there is no $q\in\mathbb{K}[x,y]$ and no $f\in\mathbb{K}[x]$ and $g\in\mathbb{K}[y]$ such that $r + qp = f - g$. If there were, then 
\begin{equation*}
\mathrm{Newt}(q) + \mathrm{Newt}(p) = \mathrm{Newt}(f - g - r).
\end{equation*}
However, $\mathrm{Newt}(f-g-r)$ cannot have a Minkowski summand of the form $\mathrm{Newt}(p)$.
If $\mathrm{Newt}(f - g - r)$ were equal to $\mathrm{Newt}(f-g)$, then there is only one edge whose outward pointing normal has only positive coordinates, conflicting with $\mathrm{Newt}(p)$ having two such edges. Hence $(1,1)$ needs to be a vertex of $\mathrm{Newt}(f-g-r)$. But this is conflicting with the slopes of the edges of $\mathrm{Newt}(p)$ whose outward pointing normals are positive.
\end{Example}

The Newton polygons of (the numerators of) $r$ and $p$ sometimes indicate that $\mathrm{F}(r,p)$ is empty. The next proposition says that so does the location of the poles of $r$ on the curve associated with $p$.

\begin{Proposition}
Assume there is a pole $(s_1,s_2)$ of $r$ on the curve associated with $p$ and two infinite cycle-free paths $v_0, v_1,\dots$ and $\tilde{v}_0, \tilde{v}_1,\dots$ that start at $(s_1,s_2)$ and do not contain any other pole of $r$. If the first coordinate of $v_1$ equals $s_1$ and the first coordinate of $\tilde{v}_1$ equals $s_2$, then $\mathrm{F}(r,p)$ is empty.
\end{Proposition}
\begin{proof}
The statement is an immediate consequence of Proposition~\ref{prop:specialPole}.
\end{proof}

\section{Acknowledgements}

Thanks go to the Johannes Kepler University Linz and the state of Upper Austria which supported this work with the grant LIT-2022-11-YOU-214. Thanks also go to the Austrian Academy of Sciences at which the author was employed while part of this work was done.

\bibliographystyle{plain}
\bibliography{sepVarCleanUp}

\end{document}